\documentclass[10pt]{amsart}

\usepackage{amsmath,amsthm,amsfonts,latexsym,amssymb,mathrsfs,setspace,enumerate}

\newtheorem{thm}{Theorem}

\newtheorem{lem}[thm]{Lemma}

\newtheorem{cor}[thm]{Corollary}
\numberwithin{thm}{section}
\numberwithin{equation}{section}

\theoremstyle{definition}

\newtheorem*{conj}{Conjecture}

\newcommand{\rat}{\mathbb Q}
\newcommand{\real}{\mathbb R}

\newcommand{\alg}{\overline\rat}
\newcommand{\algt}{\alg^{\times}}

\newcommand{\intg}{\mathbb Z}
\newcommand{\nat}{\mathbb N}

\newcommand{\X}{\mathcal X}
\newcommand{\G}{\mathcal G}

\newcommand{\Sh}{\mathcal S}
\newcommand{\tor}{\mathrm{Tor}}

\newcommand{\rad}{\mathrm{Rad}}

\newcommand{\comment}[1]{}

\title{A collection of metric Mahler measures}
\author[C.L. Samuels]{Charles L. Samuels}
\address{University of British Columbia, Department of Mathematics, 1984 Mathematics Road, Vancouver, BC V6T 1Z2, Canada \newline \indent
Simon Fraser University, Department of Mathematics, 8888 University Drive, Burnaby, BC V5A 1S6, Canada}
\email{csamuels@math.ubc.ca}
\subjclass[2000]{Primary 11R04; Secondary 26A48}
\keywords{Weil height, Mahler measure, metric Mahler measure, Lehmer's problem}

\begin{document}

\begin{abstract}
	Let $M(\alpha)$ denote the Mahler measure of the algebraic number $\alpha$.  In a recent paper, Dubickas and Smyth constructed a metric version of the Mahler measure 
	on the multiplicative group of algebraic numbers.  Later, Fili and the author used similar techniques to study a non-Archimedean version.  We show how to generalize 
	the above constructions in order to associate, to each point in $(0,\infty]$, a metric version $M_x$ of the Mahler measure, each having a triangle inequality
	of a different strength.  We are able to compute $M_x(\alpha)$ for sufficiently small $x$, identifying, in the process, a function $\bar M$ with certain minimality
	properties.  Further, we show that the map $x\mapsto M_x(\alpha)$ defines a continuous function on the positive real numbers.
\end{abstract}

\maketitle

\section{Introduction} \label{Intro}

Let $f$ be a polynomial with complex coefficients given by
\begin{equation*}
	f(z) = a\cdot\prod_{n=1}^N(z-\alpha_n).
\end{equation*}
We define the {\it (logarithmic) Mahler measure} $M$ of $f$ by
\begin{equation*}
	M(f) = \log|a|+\sum_{n=1}^N\log^+|\alpha_n|.
\end{equation*}
If $\alpha$ is a non-zero algebraic number, we define the Mahler measure of $\alpha$ by
\begin{equation*}
	M(\alpha) = M(\min_\intg(\alpha)).
\end{equation*}
In other words, $M(\alpha)$ is simply the Mahler measure of the minimal polynomial of $\alpha$ over $\intg$.  It is well-known that
\begin{equation} \label{MahlerInverses}
	M(\alpha) = M(\alpha^{-1})
\end{equation}
for all algebraic numbers $\alpha$.

It is a consequence of a theorem of Kronecker that $M(\alpha) = 0$ if and only if $\alpha$ is a root of unity.  In a famous 1933 paper, D.H. Lehmer \cite{Lehmer} 
asked whether there exists a constant $c>0$ such that $M(\alpha) \geq c$ in all other cases.  He could find no algebraic number with Mahler measure smaller than 
that of
\begin{equation*}
  \ell(x) = x^{10}+x^9-x^7-x^6-x^5-x^4-x^3+x+1,
\end{equation*}
which is approximately $0.16\ldots$.  Although the best known general lower bound is
\begin{equation*}
	M(\alpha) \gg \left(\frac{\log\log\deg\alpha}{\log\deg\alpha}\right)^3,
\end{equation*}
due to Dobrowolski \cite{Dobrowolski},
uniform lower bounds haven been established in many special cases (see \cite{BDM, Schinzel, Smyth}, for instance).  Furthermore, numerical evidence provided by 
Mossinghoff \cite{Moss, MossWeb} and Mossinghoff, Pinner and Vaaler \cite{MPV} suggests there does, in fact, exist such a constant $c$.  
This leads to the following conjecture, which we will now call Lehmer's conjecture.

\begin{conj}[Lehmer's conjecture]
	There exists a real number $c > 0$ such that if $\alpha\in\algt$ is not a root of unity then $M(\alpha) \geq c$.
\end{conj}

In an effort to create a geometric structure on the multiplicative group of algebraic numbers $\algt$, Dubickas and Smyth \cite{DubSmyth2} constructed
a metric version of the Mahler measure.  Let us briefly recall this construction.  Write
\begin{equation} \label{RestrictedProduct}
	\X(\algt) = \{(\alpha_1,\alpha_2,\ldots): \alpha_n = 1\ \mathrm{for\ all\ but\ finitely\ many}\ n\}
\end{equation}
to denote the restricted infinite direct product of $\algt$.  Let $\tau:\X(\algt) \to \algt$ be defined by 
\begin{equation*}
	\tau(\alpha_1,\alpha_2,\cdots) = \prod_{n=1}^\infty \alpha_n
\end{equation*} 
and note that $\tau$ is indeed a group homomorphism.  The {\it metric Mahler measure} $M_1$ of $\alpha$ is given by
\begin{equation*}
	M_1(\alpha) = \inf\left\{ \sum_{n=1}^\infty M(\alpha_n): (\alpha_1,\alpha_2,\ldots)\in \tau^{-1}(\alpha)\right\}.
\end{equation*}
We note that the infimum in the definition of $M_1(\alpha)$ is taken over all ways of writing $\alpha$ as a product of elements in $\algt$.
As a result of this construction, the function $M_1$ satisfies that triangle inequality
\begin{equation}  \label{MahlerTriangle}
	M_1(\alpha\beta) \leq M_1(\alpha) + M_1(\beta)
\end{equation}
for all $\alpha,\beta\in \algt$.  It can be shown that $M_1(\alpha) = 0$ if and only if $\alpha$ is a root of unity, and moreover,
$M_1$ is well-defined on the quotient group $\G = \algt/\tor(\algt)$.  Using \eqref{MahlerInverses} and \eqref{MahlerTriangle}, we find that the 
map $(\alpha,\beta)\mapsto M_1(\alpha\beta^{-1})$ is a metric on $\G$.  It is noted in \cite{DubSmyth2} that this map yields the discrete topology
if and only if Lehmer's conjecture is true.

Following the strategy of \cite{DubSmyth2}, Fili and the author \cite{FiliSamuels} examined a non-Archimedean version of the metric Mahler measure.
That is, define the {\it ultrametric Mahler measure} $M_\infty$ of $\alpha$ by
\begin{equation*}
	M_\infty(\alpha) = \inf\left\{ \max_{n\geq 1} M(\alpha_n): (\alpha_1,\alpha_2,\ldots)\in\tau^{-1}(\alpha)\right\},
\end{equation*}
replacing the sum in the definition of $M_1$ by a maximum.  In this case, $M_\infty$ has the strong triangle inequality
\begin{equation*}
	M_1(\alpha\beta) \leq \max \{M_1(\alpha),M_1(\beta)\}
\end{equation*}
for all $\alpha,\beta\in \algt$.  Once again, we are able to verify that
$M_\infty$ is well-defined on $\G$.  Here, the map $(\alpha,\beta)\mapsto M_\infty(\alpha\beta^{-1})$ yields a non-Archimedean metric on $\G$ which induces
the discrete topology if and only if Lehmer's conjecture is true.

In view of the definitions of $M_1$ and $M_\infty$, it is natural to define a collection of intermediate metric Mahler measures in the following way.
If $x\in (0,\infty]$, we define $M_x:\X(\algt)\to [0,\infty)$ by
\begin{equation*}
	M_x(\alpha_1,\alpha_2,\ldots) = \left\{
		\begin{array}{ll}
			\displaystyle \left( \sum_{n=1}^{\infty} M(\alpha_n)^x\right)^{1/x} & \mathrm{if}\ x\in (0,\infty) \\
			& \\
			\displaystyle \max_{n\geq 1}\{M(\alpha_n)\} & \mathrm{if}\ x = \infty.
		\end{array}
		\right.
\end{equation*}
In the case that $x \geq 1$, we see that $M_x(\alpha_1,\alpha_2,\ldots)$ is the $L^x$ norm on the vector $(M(\alpha_1),M(\alpha_2),\ldots)$.
Then we define the {\it $x$-metric Mahler measure} by
\begin{equation} \label{xMetricMahlerDef}
	M_x(\alpha) = \inf\{M_x(\bar\alpha): \bar\alpha\in\tau^{-1}(\alpha)\}
\end{equation}
and note that this definition generalizes those of $M_1$ and $M_\infty$.  Indeed, the $1$- and $\infty$-metric Mahler measures are simply the
metric and ultrametric Mahler measures, respectively.

In \cite{DubSmyth2}, Dubickas and Smyth showed that if Lehmer's conjecture is true, then the infimum in the definition of $M_1(\alpha)$ must
always be achieved.  The author \cite{Samuels} was able to verify that the infima in $M_1(\alpha)$ and $M_\infty(\alpha)$ are achieved even without the assumption of 
Lehmer's conjecture.  Moreover, this infimum must always be attained in a relatively simple subgroup of $\algt$.  In particular, if $K$ is a number field we write
\begin{equation*}
	\rad(K) = \left\{\alpha\in\algt:\alpha^r\in K\mathrm{\ for\ some}\ r\in\nat\right\}.
\end{equation*}
For any algebraic number $\alpha$, let $K_\alpha$ denote the Galois closure of $\rat(\alpha)$ over $\rat$.
We showed in \cite{Samuels} that the infimum in both $M_1(\alpha)$ and $M_\infty(\alpha)$ is always attained by some 
\begin{equation*}
	\bar\alpha \in \tau^{-1}(\alpha) \cap \X(\rad(K_\alpha)).
\end{equation*}
where $\X(\rad(K_\alpha))$ is defined similarly to $\X(\algt)$ in \eqref{RestrictedProduct}.
Not surprisingly, the same argument can be used to establish the analog for all values of $x$.

\begin{thm} \label{Achieved}
	Suppose $\alpha$ is a non-zero algebraic number and $x\in (0,\infty]$.  Then there exists a point $\bar\alpha\in \tau^{-1}(\alpha) \cap \X(\rad(K_\alpha))$ such that
	$M_x(\alpha) = M_x(\bar\alpha)$.
\end{thm}

We now turn our attention momentarily to the computation of some values of $M_x(\alpha)$.  First define
\begin{equation*}
	C(\alpha) = \inf\{ M(\gamma): \gamma\in K_\alpha\setminus \tor(\algt)\}
\end{equation*}
and note that by Northcott's Theorem \cite{Northcott}, the infimum on the right hand side of this definition is always achieved.  In paricular, this means
that $C(\alpha) > 0$.

The author \cite{Samuels2} gave a strategy for reducing the computation of $M_\infty(\alpha)$ to a finite set.  
The method uses the {\it modified Mahler measure}
\begin{equation} \label{MBarDef}
	\bar M(\alpha) = \inf \{M(\zeta\alpha):\zeta\in \tor(\algt)\}
\end{equation}
and gives the value of $M_\infty$ in terms of $\bar M$.  Although $\bar M$ requires taking an infimum over an infinite set, it is often very reasonable to calculate.
Indeed, the infimum on the right hand side of \eqref{MBarDef} is always attained at a root of unity $\zeta$ that makes $\deg(\zeta\alpha)$ as small as possible.
This function $\bar M$ arises again when computing $M_x(\alpha)$ for small $x$ in a more straightforward way than in \cite{Samuels2}.

\begin{thm} \label{SmallP}
	If $\alpha$ is a non-zero algebraic number and $x$ is a positive real number satisfying 
	\begin{equation} \label{AlwaysX}
		x\cdot (\log \bar M(\alpha) - \log C(\alpha)) \leq \log 2
	\end{equation}
	then $M_x(\alpha) = \bar M(\alpha)$.
\end{thm}

As we will discuss in detail in section \ref{AbelianHeights}, the construction given by \eqref{xMetricMahlerDef} is not unique to the Mahler measure.  
Suppose $\phi:\algt\to [0,\infty)$ satisfies 
\begin{equation} \label{BasicHeightProps}
	\phi(1) = 0\quad\mathrm{and}\quad \phi(\alpha) = \phi(\alpha^{-1})\ \mathrm{for\ all}\ \alpha\in\algt,
\end{equation}
and write
\begin{equation*}
	\phi_x(\alpha_1,\alpha_2,\ldots) = \left\{
		\begin{array}{ll}
			\displaystyle \left( \sum_{n=1}^{\infty} \phi(\alpha_n)^x\right)^{1/x} & \mathrm{if}\ x\in (0,\infty) \\
			& \\
			\displaystyle \max_{n\geq 1}\{\phi(\alpha_n)\} & \mathrm{if}\ x = \infty.
		\end{array}
		\right.
\end{equation*}
Generalizing the metric Mahler measure, let $\phi_x$ be defined by
\begin{equation} \label{xMetricSwitchDef}
	\phi_x(\alpha) = \inf\{\phi_x(\bar\alpha): \bar\alpha\in\tau^{-1}(\alpha)\}.
\end{equation}
We now write $\Sh(M)$ to denote the set of all functions $\phi$ satisfying \eqref{BasicHeightProps} such that $\phi_x(\alpha) = M_x(\alpha)$ for all $\alpha\in\algt$ 
and $x\in (0,\infty]$.  We are able to show that $\bar M$ belongs to $\Sh(M)$.  Moreover, it is a consequence of Theorem \ref{SmallP} that $\bar M$ is the minimal element 
of $\Sh(M)$.

\begin{cor} \label{MBarMinimal}
	We have that $\bar M\in \Sh(M)$.  Moreover, if $\psi\in\Sh(M)$ then $\psi(\alpha) \geq \bar M(\alpha)$ for all $\alpha\in\algt$.
\end{cor}

We now ask if the map $x\mapsto M_x(\alpha)$ is continuous on $\real_{>0}$ for every algebraic number $\alpha$.  We recall that Theorem \ref{Achieved} asserts that, 
for each $x$, there exists 
a point $\bar\alpha\in \tau^{-1}(\alpha)$ that attains the infimum in the definition of $M_x(\alpha)$.  If the infimum is achieved at the same point $(\alpha_1,\alpha_2,\ldots)$
for all real $x$, then we have that
\begin{equation*}
	M_x(\alpha) = \left( \sum_{n=1}^N M(\alpha_n)^x\right)^{1/x}
\end{equation*}
which clearly defines a continuous function.  Unfortunately, using the example of $M_x(p^2)$ for a rational prime $p$, we see that this is not the case.

\begin{thm} \label{NotUniform}
	Let $p$ be a rational prime and assume that $(\alpha_1,\alpha_2,\ldots)\in \tau^{-1}(p^2)$ with $M_x(p^2) = M_x(\alpha_1,\alpha_2,\cdots)$.
	\begin{enumerate}[(i)]
		\item\label{P2Small} If $x\cdot (\log\log (p^2) - \log\log 2) < \log 2$ then precisely one point $\alpha_n$ differs from a root of unity.
		\item\label{P2Large} If $x >1$ then at least two points $\alpha_n$ differ from a root of unity.
		\end{enumerate}
\end{thm}

Although the infimum in $M_x(\alpha)$ is not achieved at the same point for all $x$, we are able to prove that $x\mapsto M_x(\alpha)$ is continuous for all $\alpha$.

\begin{thm} \label{Continuous}
	If $\alpha$ is a non-zero algebraic number then the map $x \mapsto M_x(\alpha)$ is continuous on the positive real numbers.
\end{thm}

It is worth noting that continuity appears to be somewhat special to the Mahler measure.  That is, we cannot expect an arbitrary function $\phi$ satisfying \eqref{BasicHeightProps} 
to be such that $x\mapsto \phi_x(\alpha)$ is continuous.  Even making a slight modification to the Mahler measure causes continuity to fail.  For example, define
the {\it Weil height} of $\alpha\in \algt$ by
\begin{equation*}
	h(\alpha) = \frac{M(\alpha)}{\deg\alpha}
\end{equation*}
and note that, in view of our remarks about the Mahler measure, $h(\alpha) = 0$ if and only if $\alpha$ is a root of unity.  In fact, it is well-known that
\begin{equation} \label{WeilHeightDefined}
	h(\alpha) = h(\zeta\alpha)
\end{equation}
for all roots of unity $\zeta$.  Moreover, we have that $h(\alpha) = h(\alpha^{-1})$ for all $\alpha\in\algt$ so that $h$ satisfies \eqref{BasicHeightProps}.  
Unlike the Mahler measure, we know how to compute $h_x(\alpha)$ for every $x$ and $\alpha$.

\begin{thm} \label{WeilHeightComp}
	If $\alpha$ is a non-zero algebraic number then
	\begin{equation*}
		h_x(\alpha) = \left\{ \begin{array}{ll}
			h(\alpha) & \mathrm{if}\ x \leq 1 \\
			0 & \mathrm{if}\ x > 1.
		\end{array}
		\right.
\end{equation*}
\end{thm}

As we have noted, Theorem \ref{WeilHeightComp} does indeed show that $x\mapsto h_x(\alpha)$ is possibly discontinuous.  More specifically, it is continuous if and only if
$\alpha$ is a root of unity.

\section{Heights on Abelian groups} \label{AbelianHeights}

In this section, we generalize our $x$-metric Mahler measure construction to a very broad class of functions on an abelian group $G$ by exploring
definition \eqref{xMetricSwitchDef} in more detail.
We are able to establish some basic properties in this situation that we can use to prove our main results.

Let $G$ be a multiplicatively written abelian group.  We say that $\phi:G \to [0,\infty)$ is a {\it (logarithmic) height} on $G$ if
\begin{enumerate}[(i)]
	\item $\phi(1) = 0$, and
	\item $\phi(\alpha) = \phi(\alpha^{-1})$ for all $\alpha\in G$.
\end{enumerate}
If $\psi$ is another height on $G$, we follow the conventional notation that
\begin{equation*}
	\phi = \psi \quad \mathrm{or} \quad \phi \leq \psi
\end{equation*}
when $\phi(\alpha) = \psi(\alpha)$ or $\phi(\alpha) \leq \psi(\alpha)$ for all $\alpha\in G$, respectively.  We write
\begin{equation*}
	Z(\phi) = \{ \alpha\in G: \phi(\alpha) = 0\}
\end{equation*}
to denote the {\it zero set} of $\phi$.

If $x$ is a positive real number then we say that $\phi$ has the {\it $x$-triangle inequality} if
\begin{equation*}
	\phi(\alpha\beta) \leq \left (\phi(\alpha)^x + \phi(\beta)^x\right )^{1/x}
\end{equation*}
for all $\alpha,\beta\in G$.  We say that $\phi$ has the {\it $\infty$-triangle inequaltiy} if
\begin{equation*}
	\phi(\alpha\beta) \leq \max\{\phi(\alpha),\phi(\beta)\}
\end{equation*}
for all $\alpha,\beta\in G$.  For appropriate $x$, we say that these functions are {\it $x$-metric heights}.   We observe that the $1$-triangle inequality is simply 
the classical triangle inequality while the $\infty$-triangle inequality is the strong triangle inequality.  We also obtain the following ordering of the $x$-triangle
inequalities.

\begin{lem} \label{Intermediates}
	Suppose that $G$ is an abelian group and that $x,y\in (0,\infty]$ with $x\geq y$.  If $\phi$ is an $x$-metric height on $G$ then
	$\phi$ is also a $y$-metric height on $G$.
\end{lem}
\begin{proof}
  If $a,b$ and $q$ are real numbers with $a,b \geq 0$ and $q\geq 1$, then it is easily verified that
  \begin{equation} \label{MVTapp}
    a^q+b^q \leq (a+b)^q.
  \end{equation}
  Let us now assume that $\phi$ has the $x$-triangle inequality and that $\alpha,\beta\in G$.  If $x=y =\infty$ then the lemma is completely trivial.  
	If $x = \infty$ and $y <\infty$ then we have that
	\begin{equation*}
		\phi(\alpha\beta) \leq \max\{\phi(\alpha),\phi(\beta)\} = \max\{\phi(\alpha)^y,\phi(\beta)^y\}^{1/y} \leq (\phi(\alpha)^y + \phi(\beta)^y)^{1/y}
	\end{equation*}
	so that the result follows easily as well.  Hence, we assume now that $\infty > x\geq y$.  In this situation, we have that $x/y \geq 1$.
	Therefore, by \eqref{MVTapp} we have that
	\begin{equation*}
		(\phi(\alpha)^y + \phi(\beta)^y)^{x/y} \geq \phi(\alpha)^x + \phi(\beta)^x
	\end{equation*}
	and it follows that
	\begin{equation*}
		(\phi(\alpha)^y + \phi(\beta)^y)^{1/y} \geq (\phi(\alpha)^x + \phi(\beta)^x)^{1/x}.
	\end{equation*}
	Hence, we have that $\phi(\alpha\beta) \leq (\phi(\alpha)^y + \phi(\beta)^y)^{1/y}$ so that $\phi$ has the $y$-triangle inequaity.
\end{proof}

We now observe that each $x$-metric height is well-defined on the quotient group $G/Z(\phi)$.  
In the case that $x\geq 1$, the map $(\alpha,\beta) \mapsto \phi(\alpha\beta^{-1})$ defines a metric on $G/Z(\phi)$.

\begin{thm} \label{MetricProperties}
	If $\phi:G\to [0,\infty)$ is an $x$-metric height for some $x\in (0,\infty]$ then
	\begin{enumerate}[(i)]
		\item\label{Subgroup} $Z(\phi)$ is a subgroup of $G$.
		\item\label{WellDefined} $\phi(\zeta \alpha) = \phi(\alpha)$ for all $\alpha\in G$ and $\zeta\in Z(\phi)$.  That is, $\phi$ is well-defined on the quotient $G/Z(\phi)$.
		\item\label{FancyMetric} If $x\geq 1$, then the map $(\alpha,\beta)\mapsto \phi(\alpha\beta^{-1})$ defines a metric on $G/Z(\phi)$.
	\end{enumerate}
\end{thm}
\begin{proof}
	We first establish \eqref{Subgroup}.  Obviously, we have that $1\in Z(G)$ by definition of height.  Further, if $\phi(\alpha) = 0$ 
	then again by definition of height we know that $\phi(\alpha^{-1}) = 0$.
	If $\alpha,\beta\in Z(G)$ then using the $x$ triangle inequality we obtain
	\begin{equation*}
		\phi(\alpha\beta) \leq (\phi(\alpha)^x + \phi(\beta)^x)^{1/x} = 0.
	\end{equation*}
	Therefore, $\alpha\beta\in Z(G)$ so that $Z(G)$ forms a subgroup.
	
	To prove \eqref{WellDefined}, we see that the $x$-triangle inequality yields
	\begin{align*}
		\phi(\alpha) & = \phi(\zeta^{-1}\zeta\alpha) \\
			& \leq (\phi(\zeta^{-1})^x + \phi(\zeta\alpha)^x)^{1/x} \\
			& = \phi(\zeta\alpha) \\
			& \leq (\phi(\zeta)^x + \phi(\alpha)^x)^{1/x} \\
			& = \phi(\alpha)
	\end{align*}
	implying that $\phi(\alpha) = \phi(\zeta\alpha)$.
	
	Finally, if $x\geq 1$ then Lemma \ref{Intermediates} implies that $\phi$ has the triangle inequality.  It then follows immediately that
	the map $(\alpha,\beta)\mapsto \phi(\alpha\beta^{-1})$ is a metric on $G/Z(\phi)$.
\end{proof}

We are careful to note that if $x<1$ then the map $(\alpha,\beta) \mapsto \phi(\alpha\beta^{-1})$ does not, in general, form a metric on $G/Z(\phi)$.  In this case, 
the $x$-triangle inequality is indeed weaker than the triangle inequality, so we cannot expect the above map to form a metric except in trivial cases. 

We now follow the method of Dubickas and Smyth for creating a metric from the Mahler measure.  Write
\begin{equation*}
	\X(G) = \{(\alpha_1,\alpha_2,\ldots): \alpha_n = 1\ \mathrm{for\ almost\ every}\ n\}
\end{equation*}
and, as before, let $\tau:\X(G) \to G$ be defined by 
\begin{equation*}
	\tau(\alpha_1,\alpha_2,\cdots) = \prod_{n=1}^\infty \alpha_n
\end{equation*} 
so that $\tau$ is a group homomorphism.  For each point $x\in (0,\infty]$ we define the map $\phi_x:\X(G) \to [0,\infty)$ by
\begin{equation*}
	\phi_x(\alpha_1,\alpha_2,\ldots) = \left\{
		\begin{array}{ll}
			\displaystyle \left( \sum_{n=1}^{\infty} \phi(\alpha_n)^x\right)^{1/x} & \mathrm{if}\ x\in (0,\infty) \\
			& \\
			\displaystyle \max_{n\geq 1}\{\phi(\alpha_n)\} & \mathrm{if}\ x = \infty.
		\end{array}
		\right.
\end{equation*}
Then we define the {\it $x$-metric version} of $\phi_x$ of $\phi$ by
\begin{equation*}
	\phi_x(\alpha) = \inf\{\phi_x(\bar\alpha): \bar\alpha\in\tau^{-1}(\alpha)\}.
\end{equation*}
It is immediately clear that if $\psi$ is another height on $G$ with $\phi \geq \psi$, then $\phi_x \geq \psi_x$ for all $x$.
Among other things, we see that $\phi_x$ is indeed an $x$-metric height on $G$.

\begin{thm} \label{MetricConstruction}
  If $\phi:G\to [0,\infty)$ is a height on $G$ and $x\in (0,\infty]$ then
  \begin{enumerate}[(i)]
  \item\label{MetricHeightConversion} $\phi_x$ is an $x$-metric height on $G$ with $\phi_x\leq\phi$.
  \item\label{BestMetricHeight} If $\psi$ is an $x$-metric height with $\psi\leq\phi$ then 
    $\psi\leq \phi_x$.
  \item\label{NoChangeMetric} $\phi = \phi_x$ if and only if $\phi$ is an $x$-metric height.  In particular, $(\phi_x)_x = \phi_x$.
  \item\label{Comparisons} If $y\in (0,x]$ then $\phi_y \geq \phi_x$.
  \end{enumerate}
\end{thm}
\begin{proof}
	For the proofs of \eqref{MetricHeightConversion}-\eqref{NoChangeMetric}, we will assume that $x < \infty$.  The proofs for the
	case $x = \infty$ are quite similar to the proofs for other cases so we will not include them here.  See \cite{FiliSamuels} for detailed proofs when $x=\infty$.

	To prove \eqref{MetricHeightConversion}, let $\alpha,\beta\in G$.  We observe that if $(\alpha_1,\alpha_2,\ldots)\in \tau^{-1}(\alpha)$ and
	$(\beta_1,\beta_2,\ldots)\in \tau^{-1}(\beta)$ then it is obvious that
	\begin{equation*}
		\alpha\beta = \left(\prod_{n=1}^\infty \alpha_n\right)\left(\prod_{n=1}^\infty \beta_n\right).
	\end{equation*}
	We may also write
	\begin{equation*}
		\alpha\beta = \prod_{n=1}^\infty \alpha_n\beta_n
	\end{equation*}
	implying that $\tau(\alpha_1,\beta_1,\alpha_2,\beta_2,\ldots) = \alpha\beta$.  In other words, we have that
	\begin{equation} \label{ProductInInverseImage}
		(\alpha_1,\beta_1,\alpha_2,\beta_2,\ldots) \in \tau^{-1}(\alpha\beta).
	\end{equation}
	This yields that
	\begin{align} \label{InfIncrease}
		\phi_x(\alpha\beta)^x & = \inf \{\phi_x(\gamma_1,\gamma_2,\ldots)^x: (\gamma_1,\gamma_2,\ldots)\in \tau^{-1}(\alpha\beta) \} \nonumber \\
			& = \inf\{\phi_x(\alpha_1,\beta_1,\alpha_2,\beta_2,\ldots)^x: \alpha_n,\beta_n\in G,\ (\alpha_1,\beta_1,\ldots)\in \tau^{-1}(\alpha\beta)\} \nonumber \\
			& \leq \inf\{\phi_x(\alpha_1,\beta_1,\alpha_2,\beta_2,\ldots)^x: (\alpha_1,\ldots)\in \tau^{-1}(\alpha),\ (\beta_1,\ldots)\in \tau^{-1}(\beta) \}.
	\end{align}
	We note that
	\begin{align*}
		\phi_x(\alpha_1,\beta_1,\alpha_2,\beta_2,\ldots)^x & = \sum_{n=1}^\infty \left (\phi(\alpha_n)^x + \phi(\beta_n)^x\right) \\
			& = \sum_{n=1}^\infty \phi(\alpha_n)^x + \sum_{n=1}^\infty \phi(\beta_n)^x \\
			& = \phi_x(\alpha_1,\ldots)^x + \phi_x(\beta_1,\ldots)^x.
	\end{align*}
	Then using \eqref{InfIncrease} we find that
	\begin{align*}
		\phi(\alpha\beta)^x & \leq \inf\{\phi_x(\alpha_1,\ldots)^x + \phi_x(\beta_1,\ldots)^x: (\alpha_1,\ldots)\in \tau^{-1}(\alpha),\ (\beta_1,\ldots)\in \tau^{-1}(\beta) \} \\
			& = \inf\{\phi_x(\alpha_1,\ldots)^x: (\alpha_1,\ldots)\in \tau^{-1}(\alpha)\} \\
			& \qquad + \inf\{\phi_x(\beta_1,\ldots)^x: (\beta_1,\ldots)\in \tau^{-1}(\beta)\} \\
			& = \phi_x(\alpha)^x + \phi_x(\beta)^x
	\end{align*}
	and it follows that
	\begin{equation*}
		\phi_x(\alpha\beta) \leq (\phi_x(\alpha)^x + \phi_x(\beta)^x)^{1/x}.
	\end{equation*}
	To complete the proof of \eqref{MetricHeightConversion}, we observe that $(\alpha,1,1,\ldots) \in \tau^{-1}(\alpha)$ so 
	that $\phi_x(\alpha) \leq \phi(\alpha)$ for all $\alpha\in G$.
	
	To prove \eqref{BestMetricHeight}, we note that
	\begin{align*}
		\phi_x(\alpha) & = \inf\left\{ \left(\sum_{n=1}^N\phi(\alpha_n)^x\right)^{1/x}:(\alpha_1,\alpha_2,\ldots)\in \tau^{-1}(\alpha)\right\} \\
			& \geq \inf\left\{ \left(\sum_{n=1}^N\psi(\alpha_n)^x\right)^{1/x}:(\alpha_1,\alpha_2,\ldots)\in \tau^{-1}(\alpha)\right\} \\
			& \geq \psi(\alpha)
	\end{align*}
	where the last inequality follows from the fact that $\psi$ has the $x$-triangle inequality.
	
	To prove \eqref{NoChangeMetric}, we first observe that if $\phi = \phi_x$ then clearly $\phi$ is an $x$-metric height.  If $\phi$ is already a metric height,
	then by \eqref{BestMetricHeight}, we obtain that $\phi\leq \phi_x$.  But we always have $\phi_x\leq \phi$ so the result follows.  Of course, $\phi_x$ is an $x$-metric height
	so this yields immediately $\phi_x = (\phi_x)_x$.
	
	To establish \eqref{Comparisons}, we see that
	\begin{align*}
		\phi_y(\alpha) & = \inf\left\{ \left(\sum_{n=1}^N\phi(\alpha_n)^y\right)^{1/y}:(\alpha_1,\alpha_2,\ldots)\in \tau^{-1}(\alpha)\right\} \\
			& = \inf\left\{ \left(\sum_{n=1}^N\phi(\alpha_n)^y\right)^{\frac{x}{y}\cdot\frac{1}{x}}:(\alpha_1,\alpha_2,\ldots)\in \tau^{-1}(\alpha)\right\}.
	\end{align*}
	But we have that $x\geq y$ so that $x/y \geq 1$.  Therefore, by Lemma \ref{MVTapp} we have that
	\begin{equation*}
		\left(\sum_{n=1}^N\phi(\alpha_n)^y\right)^{x/y} \geq \sum_{n=1}^N\phi(\alpha_n)^x
	\end{equation*}
	which yields $\phi_y(\alpha) \geq \phi_x(\alpha)$.
\end{proof}

For a given height $\phi$ on $G$, let $\Sh(\phi)$ denote the set of all heights $\psi$ on $G$ such that $\psi_x = \phi_x$ for all $x\in (0,\infty]$.  Further, define
the height $\phi_0$ by
\begin{equation} \label{OptimalFunction}
	\phi_0(\alpha) = \lim_{x\to 0^+} \phi_x(\alpha).
\end{equation}
By \eqref{MetricHeightConversion} of Theorem \ref{MetricConstruction}, we know that $\phi_x \leq \phi$ for all $x$.  Moreover, \eqref{Comparisons} of the same theorem 
states that $x\mapsto \phi_x(\alpha)$ is non-increasing.  This means that the limit on the right hand side of \eqref{OptimalFunction} does 
indeed exist and
\begin{equation} \label{BoundForOptimal}
	\phi_0 \geq \phi_x
\end{equation}
for all $x\in (0,\infty]$.  We now observe that $\phi_0$ is the minimal element of $\Sh(\phi)$.

\begin{thm} \label{OptimalMinimal}
	If $\phi$ is a height on $G$ then $\phi_0\in\Sh(\phi)$.  Moreover, if $\psi\in \Sh(\phi)$ then $\psi \geq \phi_0$.
\end{thm}
\begin{proof}
	As we have noted, $\phi_0\geq \phi_x$ for all $x$.  Hence, we obtain immediately that $(\phi_0)_x \geq (\phi_x)_x = \phi_x$.  On the other hand,
	we know that $\phi_x \leq \phi$ so that
	\begin{equation*}
		\phi_0(\alpha) = \lim_{x\to 0^+} \phi_x(\alpha) \leq \phi(\alpha)
	\end{equation*}
	for all $\alpha\in G$.  In other words, we have that $\phi_0 \leq \phi$ so that $(\phi_0)_x \leq \phi_x$ establishing the first statement of the theorem.
	
	To prove the second statement, assume that $\psi\in \Sh(\phi)$ so that $\phi_x = \psi_x$ for all $x$.  Hence we have that
	\begin{equation*}
		\phi_0(\alpha) = \lim_{x\to 0^+}\phi_x(\alpha) = \lim_{x\to 0^+}\psi_x(\alpha) \leq \psi(\alpha)
	\end{equation*}
	for all $\alpha\in G$ verifying the theorem.
\end{proof}

We now define the {\it modified version} of $\phi$ by
\begin{equation*}
	\bar\phi(\alpha) = \inf \{\phi(\zeta\alpha): \zeta\in Z(\phi)\}.
\end{equation*}
In the case of the Mahler measure, we have stated in the introduction that $\bar\phi = \phi_0$.  However, in the general case, we can conclude only that 
$\bar\phi$ belongs to $\Sh(\phi)$.

\begin{thm} \label{BarPhiMetrics}
	If $\phi$ is a height on $G$ then $\bar\phi\in\Sh(\phi)$.
\end{thm}
\begin{proof}
	We must show that $\bar\phi_x = \phi_x$ for all $x\in (0,\infty]$.  Since $1\in Z(\phi)$, we have immediately that $\bar\phi \leq \phi$, which means that
	\begin{equation*}
		\bar\phi_x \leq \phi_x.
	\end{equation*}
	Now for any $\alpha\in G$, we have that
	\begin{equation*}
		\phi_x(\alpha) \leq \inf \{(\phi(\zeta^{-1})^x + \phi(\zeta\alpha)^x)^{1/x}:\zeta\in Z(\phi)\} = \inf \{\phi(\zeta\alpha):\zeta\in Z(\phi)\} = \bar\phi(\alpha)
	\end{equation*}
	implying that $\phi_x \leq \bar\phi$.  Then taking $x$-metric versions and using \eqref{NoChangeMetric} of Theorem \ref{MetricConstruction} we find that
	\begin{equation*}
		\phi_x = (\phi_x)_x \leq \bar\phi_x
	\end{equation*}
	completing the proof.
\end{proof}

We may now ask what we can say about the map $x\mapsto \phi_x(\alpha)$ for fixed $\phi$ and $\alpha$.  
As we have noted, this map is non-increasing for all $\alpha$.  Since $\phi_x(\alpha)$ is bounded
from above and below by constants not depending on $x$, both left and right hand limits exist at every point.  Moreover, we always have
\begin{equation*}
	\lim_{x\to \bar x^-}\phi_x(\alpha) \geq \phi_{\bar x}(\alpha)  \geq \lim_{x\to \bar x^+}\phi_x(\alpha)
\end{equation*}
when $\bar x >0$.  We say that a map $f:\real\to\real$ is {\it left} or {\it right semi-continuous} at a point $\bar x\in \real$ if
\begin{equation*}
	\lim_{x\to \bar x^-}f(x) = f(\bar x)\quad\mathrm{or}\quad \lim_{x\to \bar x^+}f(x) = f(\bar x),
\end{equation*}
respectively.  Indeed, $f$ is continuous at $\bar x$ if and only if $f$ is both left and right semi-continuous at $\bar x$.  Although it is a consequence of Theorem
\ref{WeilHeightComp} that $x\mapsto \phi_x(\alpha)$ is not continuous in general, we can prove the following partial result.

\begin{thm} \label{LeftSemiContinuous}
	If $\phi$ is a height on $G$ and $\alpha\in G$, then the map $x \mapsto \phi_x(\alpha)$ is left semi-continous on the positive real numbers.
\end{thm}
\begin{proof}
	We already know that $\lim_{x\to \bar x^-}\phi_x(\alpha) \geq \phi_{\bar x}(\alpha)$ so we assume that $$\lim_{x\to \bar x^-}\phi_x(\alpha) > \phi_{\bar x}(\alpha).$$
	Therefore, there exists $\varepsilon >0$ such that
	\begin{equation} \label{EpsilonSqueeze}
		\lim_{x\to \bar x^-}\phi_x(\alpha) > \phi_{\bar x}(\alpha) + \varepsilon.
	\end{equation}
	By definition of $\phi_{\bar x}$, we may choose points $\alpha_1,\ldots,\alpha_N \in G$ such that $\alpha = \alpha_1\cdots\alpha_N$ and
	\begin{equation*} \label{CloseEnough}
		\phi_{\bar x}(\alpha) + \varepsilon \geq \left( \sum_{n=1}^N \phi(\alpha_n)^{\bar x}\right)^{1/\bar x},
	\end{equation*}
	and define the function $f_\varepsilon$ by
	\begin{equation*}
		f_\varepsilon(x) = \left( \sum_{n=1}^N \phi(\alpha_n)^{x}\right)^{1/x}.
	\end{equation*}
	This yields
	\begin{equation} \label{FApprox}
		f_\varepsilon(\bar x) \leq \phi_{\bar x}(\alpha) + \varepsilon\quad\mathrm{and}\quad f_\varepsilon(x) \geq \phi_x(\alpha)\ \mathrm{for\ all}\ x.
	\end{equation}
	Also, since $f_\varepsilon$ is continuous, we have that
	\begin{equation} \label{FCont}
		f_\varepsilon(\bar x) = \lim_{x\to\bar x^-} f_\varepsilon(x).
	\end{equation}
	Combining \eqref{EpsilonSqueeze}, \eqref{FApprox} and \eqref{FCont} we obtain that
	\begin{equation*}
		f_\varepsilon(\bar x) = \lim_{x\to\bar x^-} f_\varepsilon(x) \geq \lim_{x\to\bar x^-} \phi_x(\alpha) > \phi_{\bar x}(\alpha) + \varepsilon \geq f_\varepsilon(\bar x)
	\end{equation*}
	which is a contradiction.
\end{proof}

\section{The Inifimum in $M_x(\alpha)$} \label{AchievedProofs}

Our proof of Theorem \ref{Achieved} will require the use of two results from \cite{Samuels}.  The first of these is Theorem 2.1 of \cite{Samuels}, which shows that for any
point $\bar\alpha\in \tau^{-1}(\alpha)$, there exists another point $\bar\beta\in\tau^{-1}(\alpha)\cup\X(\rad(K_\alpha))$ which has pointwise smaller
Mahler measures.  We state the Theorem using the notation of \cite{Samuels}.

\begin{thm} \label{Reduction}
	If $\alpha,\alpha_1,\ldots,\alpha_N$ are non-zero algebraic numbers with $\alpha = \alpha_1\cdots\alpha_N$ then
	there exists a root of unity $\zeta$ and algebraic numbers $\beta_1,\ldots,\beta_N$ satifying
	\begin{enumerate}[(i)]
		\item $\alpha = \zeta\beta_1\cdots\beta_N$,
		\item $\beta_n\in \rad(K_\alpha)$ for all $n$,
		\item $M(\beta_n) \leq M(\alpha_n)$ for all $n$.
	\end{enumerate}
\end{thm}

In view of Theorem \ref{Reduction}, for each $x$, we need only consider only points $\bar\alpha\in\tau^{-1}(\alpha)\cup\X(\rad(K_\alpha))$ in the definition of $M_x(\alpha)$.
In other words, in the case of $x <\infty$, the definition of $M_x(\alpha)$ may be rewritten
\begin{equation} \label{xMetricAltDef}
	M_x(\alpha) = \inf\left\{\left(\sum_{n=1}^\infty M(\alpha_n)^x\right)^{1/x}: (\alpha_1,\alpha_2,\ldots) \in \tau^{-1}(\alpha)\cup\X(\rad(K_\alpha))\right\}.
\end{equation}
Similar remarks apply in the case that $x = \infty$.
Therefore, it will be useful to have some control of the Mahler measures in the subgroup $\rad(K_\alpha)$.  For this purpose, we borrow Lemma 3.1 of \cite{Samuels}.

\begin{lem} \label{HeightInK}
	Let $K$ be a Galois extension of $\rat$.  If $\gamma\in\rad(K)$ then there exists a root of unity $\zeta$ and $L,S\in\nat$
	such that $\zeta\gamma^L\in K$ and
	\begin{equation*}
		M(\gamma) = M(\zeta\gamma^L)^S.
	\end{equation*}
	In particular, the set
	\begin{equation*}
		\{M(\gamma):\gamma\in\rad(K),\ M(\gamma) \leq B\}
	\end{equation*}
	is finite for every $B \geq 0$.
\end{lem}

It is an easy consequence of Lemma \ref{HeightInK} that $M(\gamma)$ is bounded below by the Mahler measure of an element in $K$.  Indeed, we have that
\begin{equation*}
	M(\gamma) = M(\zeta\gamma^L)^S \geq M(\zeta\gamma^L)
\end{equation*}
and $\zeta\gamma^L\in K$.  In particular, we recall that $C(\alpha)$ denotes the minimum Mahler measure in the field $K_\alpha$.  We now see easily that
\begin{equation} \label{RadBound}
	M(\gamma) \geq C(\alpha)
\end{equation}
for all $\gamma\in \rad(K_\alpha)\setminus\tor(\algt)$.  We are now prepared to prove Theorem \ref{Achieved}.

\begin{proof}[Proof of Theorem \ref{Achieved}]
	By the results of \cite{Samuels}, we know that the theorem holds for $x = \infty$, so we may assume that $x <\infty$.  Further, select a real number $B > M_x(\alpha)$.
	In view of Theorem \ref{Reduction}, we know that $M_x(\alpha)$ is the infimum of
	\begin{equation} \label{FiniteHope}
		\left(\sum_{n=1}^N M(\alpha_n)^x\right)^{1/x}
	\end{equation}
	over the set of all $N\in \nat$ and all points $\alpha_1,\ldots,\alpha_N\in \algt$ such that
	\begin{enumerate}[(i)]
		\item\label{Product} $\alpha = \alpha_1\cdots\alpha_N$,
		\item\label{Unity} At most one point $\alpha_n$ is a root of unity,
		\item\label{Rad} $\alpha_n\in \rad(K_\alpha)$ for all $n$, and
		\item\label{BBound} $\left(\sum_{n=1}^N M(\alpha_n)^x\right)^{1/x} \leq B$.
	\end{enumerate}
	We will show that the set of all values of \eqref{FiniteHope} is finite for $\alpha_1,\ldots,\alpha_N$ satisfying conditions \eqref{Product}-\eqref{BBound}.
	
	We must first give an upper bound on $N$.  We know that at least $N-1$ of the points $\alpha_1,\ldots,\alpha_N$ are not roots of unity.  For all such points, we
	have that
	\begin{equation*}
		M(\alpha_n) \geq C(\alpha)
	\end{equation*}
	by \eqref{RadBound}.  Combining this with \eqref{BBound}, we obtain that
	\begin{equation*}
		B \geq \left(\sum_{n=1}^N M(\alpha_n)^x\right)^{1/x} \geq (N-1)^{1/x} C(\alpha)
	\end{equation*}
	which yields
	\begin{equation} \label{NUpperBound}
		N \leq 1+ \left(\frac{B}{C(\alpha)}\right)^x.
	\end{equation}
	
	Also by \eqref{BBound}, it follows that $M(\alpha_n) \leq B$ for all $n$.  Moreover, since $\alpha_n\in\rad(K_\alpha)$, the second statement of Lemma \ref{HeightInK} 
	implies that there are only finitely many possible values for $M(\alpha_n)$ for each $n$.  Since $N$ is bounded above by the right hand side of \eqref{NUpperBound},
	it follows that there are only finitely many possible values for \eqref{FiniteHope}	with $\alpha_1,\ldots,\alpha_N$ satisfying \eqref{Product}-\eqref{BBound}.  
	We now know that $M_x(\alpha)$ is an infimum over a finite set, so the infimum must be achieved.
\end{proof}

\section{Minimality of $\bar M$} \label{MinimalBarM}

We first give the proof of Theorem \ref{SmallP} showing that $M_x(\alpha) = \bar M(\alpha)$ for sufficiently small values of $x$.

\begin{proof}[Proof of Theorem \ref{SmallP}]
	By Theorem \ref{BarPhiMetrics}, we have immediately that $M_x(\alpha) = \bar M_x(\alpha)$ for all $x$, so it follows that
	\begin{equation} \label{EasyUpperBound}
		M_x(\alpha) \leq \bar M(\alpha).
	\end{equation}
	Now we must prove the opposite inequality.
	
	We know by Theorem \ref{Achieved} that there exist points $\alpha_1,\ldots,\alpha_N\in \rad(K_\alpha)$ such that
	\begin{equation*}
		\alpha = \alpha_1\cdots\alpha_N\quad\mathrm{and}\quad M_x(\alpha) = \left( \sum_{n=1}^N M(\alpha_n)^x\right)^{1/x}.
	\end{equation*}
	We know that $\alpha$ is not a root of unity, so at least one of $\alpha_1,\ldots,\alpha_N$ is not a root of unity.
	
	We now consider two cases.  First, assume that precisely one of $\alpha_1,\ldots,\alpha_N$ is not a root of unity.  In other words, there exists
	a root of unity $\zeta$ and a point $\beta\in \rad(K_\alpha)\setminus \tor(\algt)$ such that $\alpha = \zeta\beta$ and
	\begin{equation*}
		M_x(\alpha) = M(\beta).
	\end{equation*}
	Of course, we also have $\beta = \alpha\zeta^{-1}$ so that
	\begin{equation*}
		\bar M(\alpha) \leq M(\alpha\zeta^{-1}) = M(\beta) = M_x(\alpha).
	\end{equation*}
	Combining this inequality with \eqref{EasyUpperBound}, the result follows.
	
	Next, assume that at least two of $\alpha_1,\ldots,\alpha_N$ are not a roots of unity.  By Lemma \ref{HeightInK}, we know that 
	$M(\alpha_n) \geq C(\alpha)$ whenever $\alpha_n$ is not a root of unity.  Hence, we obtain that
	\begin{equation*}
		M_x(\alpha) = \left(\sum_{n=1}^N M(\alpha_n)^x\right)^{1/x} \geq (2C(\alpha)^x)^{1/x}
	\end{equation*}
	so that
	\begin{equation} \label{TwoBound}
		M_x(\alpha) \geq 2^{1/x} C(\alpha).
	\end{equation}
	By our assumption, we have that
	\begin{equation*}
		\frac{1}{x} \geq \frac{\log \bar M(\alpha) - \log C(\alpha)}{\log 2}
	\end{equation*}
	which implies that
	\begin{align*}
		2^{1/x} & \geq 2^{\frac{\log \bar M(\alpha) - \log C(\alpha)}{\log 2}} \\
			& = \exp(\log \bar M(\alpha) - \log C(\alpha)) \\
			& = \frac{\exp(\log \bar M(\alpha))}{\exp(\log C(\alpha))} \\
			& = \frac{\bar M(\alpha)}{C(\alpha)}.
	\end{align*}
	It now follows from \eqref{TwoBound} that
	\begin{equation*}
		M_x(\alpha) \geq \bar M(\alpha)
	\end{equation*}
	completing the proof.
\end{proof}

Next, we establish Corollary \ref{MBarMinimal} showing that $\bar M$ is minimal in the set $\Sh(M)$.

\begin{proof}[Proof of Corollary \ref{MBarMinimal}]
	We observe again by Theorem \ref{BarPhiMetrics} that $\bar M\in \Sh(M)$.  By Theorem \ref{SmallP}, for all sufficiently small $x$, we have that 
	$\bar M(\alpha) = M_x(\alpha)$.  Hence, it follows that that
	\begin{equation*}
		\bar M(\alpha) = \lim_{x\to 0^+} M_x(\alpha) = M_0(\alpha)
	\end{equation*}
	and the result follows from Theorem \ref{OptimalMinimal}.
\end{proof}

We begin our proof of Theorem \ref{NotUniform} by giving a slight modification to Theorem \ref{SmallP}.  More specifically, it will be useful to consider what happens
when the supposed inequality \eqref{AlwaysX} is replaced by a strict inequality.

\begin{lem} \label{StrongSmallP}
	Let $\alpha$ be a non-zero algebraic number different from a root of unity and $x$ a positive real number satisfying 
	\begin{equation*}
		x\cdot (\log \bar M(\alpha) - \log C(\alpha)) < \log 2.
	\end{equation*}
	Then any point $(\alpha_1,\alpha_2,\cdots) \in \tau^{-1}(\alpha)$ that achieves the infimum in the definition of $M_x(\alpha)$ has precisely one component
	$\alpha_n$ that is not a root of unity.
\end{lem}
\begin{proof}
 	We recall first that
 	\begin{equation} \label{MexUpper}
 		M_x(\alpha) \leq \bar M(\alpha)
 	\end{equation}
 	by Theorem \ref{BarPhiMetrics}.  Next, we note that
	\begin{equation} \label{LooseBound}
		\frac{1}{x} > \frac{\log \bar M(\alpha) - \log C(\alpha)}{\log 2}.
	\end{equation}
	Assume that $\alpha_1,\ldots,\alpha_N\in \algt$ are such that
	\begin{equation} \label{AchievedApplication}
		\alpha = \alpha_1\cdots\alpha_N\quad\mathrm{and}\quad M_x(\alpha) = \left( \sum_{n=1}^N M(\alpha_n)^x\right)^{1/x}.
	\end{equation}
	and at least two of the points  $\alpha_1,\ldots,\alpha_N$ are not roots of unity.  By Theorem \ref{Reduction}, there exists a root of unity $\zeta$ and
	points $\beta_1,\ldots,\beta_N\in \rad(K_\alpha)$ such that
	\begin{equation*}
		\alpha = \zeta\beta_1\cdots\beta_N\quad\mathrm{and}\quad M(\beta_n) \leq M(\alpha_n)
	\end{equation*}
	for all $n$.  If for any $n$ we have that $M(\beta_n) < M(\alpha_n)$, then
	\begin{equation*}
		M_x(\alpha) \leq \left( \sum_{n=1}^N M(\beta_n)^x\right)^{1/x} < \left( \sum_{n=1}^N M(\alpha_n)^x\right)^{1/x}
	\end{equation*}
	which contradicts the right hand side of \eqref{AchievedApplication}.  Therefore, we have that $M(\beta_n) = M(\alpha_n)$ for all
	$n$.  In particular, at least two of the points $\beta_1,\ldots,\beta_N$ are not roots of unity.  Furthermore, since each $\beta_n\in \rad(K_\alpha)$,
	we may apply Lemma \ref{HeightInK} to see that $M(\beta_n) \geq C(\alpha)$ whenever $\beta_n$ is not a root of unity.  This yields
	\begin{equation*}
		M_x(\alpha) = \left(\sum_{n=1}^N M(\beta_n)^x\right)^{1/x} \geq (2C(\alpha)^x)^{1/x}.
	\end{equation*}
	which implies that
	\begin{equation*}
		M_x(\alpha) \geq 2^{1/x} C(\alpha).
	\end{equation*}
	However, we now have the strict inequality \eqref{LooseBound} which gives $2^{1/x} > \bar M(\alpha)/C(\alpha)$ and 
	\begin{equation*}
		M_x(\alpha) > \bar M(\alpha)
	\end{equation*}
	contradicting \eqref{MexUpper}.  Therefore, exactly one point among $\alpha_1,\ldots,\alpha_N$ is not a root of unity.
\end{proof}

Before we prove Theorem \ref{NotUniform}, we recall our remark that $\bar M(\alpha)$ is often very reasonable to compute so that Theorem \ref{SmallP} and 
Lemma \ref{StrongSmallP} are useful in applications.  The following proof is a typical example.

\begin{proof}[Proof of Theorem \ref{NotUniform}]
	Let $\alpha = p^2$.  In order to prove \eqref{P2Small}, we wish to apply Lemma \ref{StrongSmallP}, so we must compute the values of $\bar M(\alpha)$ and $C(\alpha)$.  
	We begin by observing that
	\begin{equation*}
		\bar M(\alpha) = \inf\{M(\zeta\alpha):\zeta\in \tor(\algt)\} = \inf\{\deg(\zeta\alpha)\cdot h(\zeta\alpha):\zeta\in \tor(\algt)\}.
	\end{equation*}
	Then by \eqref{WeilHeightDefined}, we obtain that
	\begin{equation} \label{MBarM}
		\bar M(\alpha) = h(\alpha)\cdot \inf\{\deg(\zeta\alpha):\zeta\in \tor(\algt)\}.
	\end{equation}
	It is clear that the infimum on the right hand side of \eqref{MBarM} is achieved since it is an infimum over positive integers.  More specifically,
	it is achieved by a root of unity $\zeta$ that makes $\deg(\zeta\alpha)$ as small as possible.  In our case, $\alpha$ is rational, so this occurs when $\zeta = 1$ leaving
	\begin{equation} \label{AlphaUpper}
		\bar M(\alpha) = \bar M(p^2) =  M(p^2) = \log (p^2).
	\end{equation}
	In addition, we know that $K_\alpha = \rat$ so that $C(\alpha) = \log 2$ which now gives
	\begin{equation*}
		x\cdot (\log \bar M(\alpha) - \log C(\alpha)) = x\cdot (\log \log (p^2) - \log \log 2) < \log 2.
	\end{equation*}
	By Lemma \ref{StrongSmallP}, we know that any point $(\alpha_1,\alpha_2,\ldots)$ that attains the infimum in $M_x(\alpha) = M_x(p^2)$ must have precisely one point $\alpha_n$
	that is not a root of unity.  This completes the proof of \eqref{P2Small}.
	
	To prove \eqref{P2Large}, we take $x > 1$ and assume that $(\alpha_1,\alpha_2,\ldots)$ attains the infimum in the definition of $M_x(p^2)$ where are most one
	point $\alpha_n$ is different from a root of unity.  Therefore, there exists a root of unity $\zeta$ and an algebraic number $\beta$ such that 
	\begin{equation*}
		p^2 = \zeta\beta\quad\mathrm{and}\quad M_x(p^2) = M(\beta).
	\end{equation*}
	Hence we find immediately that
	\begin{equation*}
		M(\beta) = M_x(p^2) \leq (M(p)^x + M(p)^x)^{1/x} = 2^{1/x}\log p.
	\end{equation*}
	Since $x > 1$, this yields that
	\begin{equation*} \label{BetaUpper}
		M(\beta) < 2\log p.
	\end{equation*}
	On the other hand, we have that $\beta = \zeta^{-1}p^2$ so that, using \eqref{AlphaUpper}, we obtain
	\begin{equation*}
		M(\beta) = M(\zeta^{-1}p^2) \geq \bar M(p^2) = 2\log p
	\end{equation*}
	which is a contradiction.  Thus, at least two points among $(\alpha_1,\alpha_2,\ldots)$ must not be roots of unity.
\end{proof}

\section{Continuity of $x\mapsto M_x(\alpha)$} \label{ContinuitySection}

We have already proved that, for any height function $\phi$, the map $x\mapsto \phi_x(\alpha)$ is left semi-continuous.  In general, we know that 
such functions are not always right semi-continuous.  However, we are able to use Theorem \ref{Achieved} and our observations about the Mahler measure 
to establish right semi-continuity in this case.

\begin{proof}[Proof of Theorem \ref{Continuous}]
	If $\alpha$ is a root of unity, then $M_x(\alpha) = 0$ for all $x$, so we may assume that $\alpha$ is not a root of unity.
	Furthermore, we know by Theorem \ref{LeftSemiContinuous} that this map is left semi-continuous at all points, so it remains only to show that it
	is right semi-continuous.
	
	Now let $\bar x >0$ be a real number, so we must show that
	\begin{equation} \label{RightSemiEnd}
		\lim_{y\to \bar x^+} M_y(\alpha) = M_{\bar x}(\alpha).
	\end{equation}
	Since $x\mapsto M_x(\alpha)$ is decreasing, we know that the left hand side of \eqref{RightSemiEnd} exists.  Moreover, we have that
	\begin{equation} \label{HalfRightSemi}
		\lim_{y\to \bar x^+} M_y(\alpha) \leq M_{\bar x}(\alpha).
	\end{equation}
	
	Now we select a point $y\in (\bar x, \bar x +1]$.  By Theorem \ref{Achieved}, there must exist points 
	\begin{equation*}
		\alpha_1,\ldots,\alpha_N \in \rad(K_\alpha) \setminus \tor(\algt)
	\end{equation*}
	and $\zeta\in \tor(\algt)$ such that
	\begin{equation*}
		\alpha = \zeta\alpha_1\cdots\alpha_N\quad\mathrm{and}\quad M_y(\alpha) = \left( \sum_{n=1}^N M(\alpha_n)^y\right)^{1/y}.
	\end{equation*}
	Since $M_y(\alpha) \leq M(\alpha)$, we may assume without loss of generality that $M(\alpha_n) \leq M(\alpha)$ for all $n$.  Furthermore, since $\alpha$ is not a root
	of unity, we know that $N\geq 1$.  For simplicity, we write now $a_n = M(\alpha_n)$ so that
	\begin{equation*}
		M_y(\alpha) = \left( \sum_{n=1}^N a_n^y\right)^{1/y},
	\end{equation*}
	and note that by Lemma \ref{HeightInK}, we have that
	\begin{equation} \label{aLower}
		a_n \geq C(\alpha)\ \mathrm{for\ all}\ n.
	\end{equation}
	Next, we define the function $f_y$ by
	\begin{equation*}
		f_y(x) = \left( \sum_{n=1}^N a_n^x\right)^{1/x}
	\end{equation*}
	and note that $f_y$ does indeed depend on $y$ because the points $\zeta$ and $\alpha_1,\ldots,\alpha_N$ depend on $y$.  We now have immediately that
	\begin{equation} \label{fAndM}
		f_y(y) = M_y(\alpha).
	\end{equation}
	Since $\alpha = \zeta\alpha_1\cdots\alpha_N$, we know that
	\begin{equation*}
		M_{\bar x}(\alpha) \leq \left(\sum_{n=1}^N M(\alpha_n)^{\bar x}\right)^{1/{\bar x}} = \left(\sum_{n=1}^N a_n^{\bar x}\right)^{1/{\bar x}} = f_y(\bar x),
	\end{equation*}
	and therefore, we obtain that
	\begin{equation} \label{fAndM2}
		 M_{\bar x}(\alpha) \leq f_y(\bar x).
	\end{equation}
	
	We know that $a_n> 0$ for all $n$ implying that $f_y(x) > 0$ for all $x$, so we may define the function $g_y(x) = \log f_y(x)$.
	Since $f_y$ is differentiable on the positive real numbers, we know that $g_y$ is as well.   Therefore, we may apply the Mean Value Theorem to it on $[\bar x, y]$.
	Hence, there exists a point $c\in [\bar x,y]$ such that
	\begin{equation*}
		g_y'(c) = \frac{g_y(y) - g_y(\bar x)}{y-\bar x} = \frac{\log f_y(y) - \log f_y(\bar x)}{y-\bar x}
	\end{equation*}
	and it follows from \eqref{fAndM} and \eqref{fAndM2} that
	\begin{equation} \label{MVTInequality}
		g_y'(c) \leq \frac{\log M_y(\alpha) - \log M_{\bar x}(\alpha)}{y-\bar x}.
	\end{equation}
	We now wish to take limits of both sides of \eqref{MVTInequality} as $y$ tends to $\bar x$ from the right.  However, it is possible that the limit of the left hand
	side either equals $-\infty$ or does not exist as $y\to \bar x^+$.  To solve this problem, we wish to give a lower bound on $g_y'(c)$ that does not depend on $y$.
	
	For any $x >0$, we note that
	\begin{align*}
		g_y'(x) & = \frac{d}{dx}\log f_y(x) \\
			& = \frac{d}{dx} \frac{1}{x}\left( \log \sum_{n=1}^N a_n^x\right) \\
			& = \frac{1}{x^2}\left( x\cdot\frac{ \left(\sum_{n=1}^N a_n^x\log a_n\right)}{\left(\sum_{n=1}^N a_n^x\right)}
					- \log \sum_{n=1}^N a_n^x \right ).
	\end{align*}
	Then using \eqref{aLower}, we have that
	\begin{equation} \label{FirstLower}
		g_y'(x) \geq \frac{1}{x^2}\left( x\cdot \log C(\alpha) - \log \sum_{n=1}^N a_n^x \right ).
	\end{equation}
	Now we need to give an upper bound on $\sum_{n=1}^N a_n^x$.  Recall that we must have $a_n = M(\alpha_n) \leq M(\alpha)$ for all $n$.
	Therefore, we have that
	\begin{equation*}
		\sum_{n=1}^N a_n^x \leq N M(\alpha)^x.
	\end{equation*}
	But using \eqref{aLower} again, we find that
	\begin{equation*}
		M(\alpha) \geq M_y(\alpha) = \left(\sum_{n=1}^N a_n^y\right)^{1/y} \geq (N C(\alpha)^y)^{1/y} =  N^{1/y}C(\alpha).
	\end{equation*}
	We also know $C(\alpha) > 0$  and $y\in (\bar x, \bar x+1]$ so that
	\begin{equation*} \label{NUpper}
		N \leq \left( \frac{M(\alpha)}{C(\alpha)}\right)^y \leq \left( \frac{M(\alpha)}{C(\alpha)}\right)^{\bar x +1},
	\end{equation*}
	and therefore
	\begin{equation*} \label{SumUpper}
		\sum_{n=1}^N a_n^x \leq \frac{M(\alpha)^{x + \bar x +1}}{C(\alpha)^{\bar x + 1}}.
	\end{equation*}
	It now follows that
	\begin{equation*}
		-\log \sum_{n=1}^N a_n^x \geq -\log \left( \frac{M(\alpha)^{x + \bar x +1}}{C(\alpha)^{\bar x + 1}}\right).
	\end{equation*}
	Combining this with \eqref{FirstLower}, we obtain that
	\begin{equation*}
		g_y'(x) \geq\frac{1}{x^2}\left( x\cdot \log C(\alpha) - \log \left( \frac{M(\alpha)^{x + \bar x +1}}{C(\alpha)^{\bar x + 1}}\right) \right),
	\end{equation*}
	so we have shown that
	\begin{equation} \label{SecondLower}
			g_y'(x) \geq \frac{x + \bar x +1}{x^2} \log \left( \frac{C(\alpha)}{M(\alpha)} \right).
	\end{equation}
	
	For simplicity, we now write $D(\alpha,\bar x, x)$ to denote the right hand side of \eqref{SecondLower}.  As a function of $x$, it is obvious that 
	$D(\alpha,\bar x, x)$ is continuous for all $x > 0$.  Hence, we may define
	\begin{equation*}
		\mathcal D(\alpha,\bar x) = \min \{D(\alpha,\bar x, x): x\in [\bar x, \bar x +1]\}.
	\end{equation*}
	Now $\mathcal D(\alpha,\bar x)$ is the desired lower bound on $g_y'(c)$ not depending on $y$.
	
	Since $c\in [\bar x,y] \subset [\bar x,\bar x+1]$, we may apply \eqref{MVTInequality} and \eqref{SecondLower} to see that
	\begin{equation*}
		\mathcal D(\alpha,\bar x) \leq D(\alpha,\bar x, c) \leq g_y'(c) \leq \frac{\log M_y(\alpha) - \log M_{\bar x}(\alpha)}{y-\bar x}.
	\end{equation*}
	By multiplying through by $y - \bar x$, we find that
	\begin{equation} \label{AlmostDone}
		(y-\bar x) \mathcal D(\alpha,\bar x) \leq \log M_y(\alpha) - \log M_{\bar x}(\alpha)
	\end{equation}
	holds for all $y\in (\bar x, \bar x + 1]$.
	
	As we have noted, $\lim_{y\to \bar x^+} M_y(\alpha)$ exists.  Since we have assumed that $\alpha$ is not a root of unity, we conclude from Theorem \ref{Achieved}
	that $M_y(\alpha) > 0$ for all $y$.  It now follows that $\lim_{y\to \bar x^+} \log M_y(\alpha)$ also exists.  Moreover, the term $\mathcal D(\alpha,\bar x)$ is a real
	number not depending on $y$, so the left hand side of \eqref{AlmostDone} tends to zero as $y$ tends to $\bar x$ from the right.  This leaves
	\begin{align*}
		0 & = \lim_{y\to \bar x^+}((y-\bar x) \mathcal D(\alpha,\bar x)) \\
			& \leq \lim_{y\to \bar x^+} (\log M_y(\alpha) - M_{\bar x}(\alpha)) \\
			& =  \lim_{y\to \bar x^+} \log M_y(\alpha) -  \lim_{y\to \bar x^+}\log M_{\bar x}(\alpha) \\
			& =  \lim_{y\to \bar x^+} \log M_y(\alpha) - \log M_{\bar x}(\alpha),
	\end{align*}
	which yeilds
	\begin{equation*}
		\log M_{\bar x}(\alpha) \leq \lim_{y\to \bar x^+} \log M_y(\alpha)
	\end{equation*}
	so that $M_{\bar x}(\alpha) \leq \lim_{y\to \bar x^+} M_y(\alpha)$ and the result follows by combining this with \eqref{HalfRightSemi}.
	
\end{proof}

\section{Weil height}

Before we begin our proof of Theorem \ref{WeilHeightComp}, we recall that if $N$ is any integer, then it is well-known that
\begin{equation} \label{IntPowers}
	h(\alpha^N) = |N|\cdot h(\alpha)
\end{equation}
for all algebraic numbers $\alpha$.  Using this fact, we are able to proceed with our proof.

\begin{proof}[Proof of Theorem \ref{WeilHeightComp}]
	First assume that $x \leq 1$.  By \eqref{MetricHeightConversion} of Theorem \ref{MetricConstruction}, we have that $h_x(\alpha) \leq h(\alpha)$.
	But also, it is well-known that $h$ is already a $1$-metric height.  Therefore, \eqref{NoChangeMetric} of Theorem \ref{MetricConstruction} implies that
	$h_1(\alpha) = h(\alpha)$.  Then by \eqref{Comparisons} of Theorem \ref{MetricConstruction}, we conclude that $h_x(\alpha) \geq h(\alpha)$ verifying the 
	theorem in the case that $x \leq 1$.
	
	Next, we assume that $x > 1$.  Let $N$ be a positive integer and select $\beta\in\algt$ such that $\beta^N = \alpha$.  Therefore, we have that
	\begin{equation*}
		h_x(\alpha) \leq \left(\sum_{n=1}^N h(\beta)^x\right)^{1/x} = (N h(\beta)^x)^{1/x} = N^{1/x}\cdot h(\beta).
	\end{equation*}
	Then using \eqref{IntPowers} we obtain that $h(\alpha) = N\cdot h(\beta)$ which yields
	\begin{equation} \label{TightUpper}
		h_x(\alpha) \leq N^{\frac{1}{x} - 1}\cdot h(\alpha).
	\end{equation}
	Since $x > 1$, the right hand side of \eqref{TightUpper} tends to zero as $N\to\infty$ completing the proof.
\end{proof}

\section{Acknowledgment}

The author wishes to thank the Max-Planck-Institut f\"ur Mathematik where the majority of this research took place.


\begin{thebibliography}{}
\bibitem{BDM} P. Borwein, E. Dobrowolski and M.J. Mossinghoff,
  {\it Lehmer's problem for polynomials with odd coefficients}, Ann. of Math. (2) {\bf 166} (2007),  no. 2, 347--366.
\bibitem{Dobrowolski} E. Dobrowolski, {\it  On a question of Lehmer and the number of irreducible factors of a polynomial},
 Acta Arith.  {\bf 34} (1979),  no. 4, 391--401.
\bibitem{DubSmyth2}  A. Dubickas and C.J. Smyth,
  {\it On the metric Mahler measure}, J. Number Theory {\bf 86} (2001), 368--387.
\bibitem{FiliSamuels} P. Fili and C.L. Samuels, {\it On the non-Archimedean metric Mahler measure}, J. Number Theory, {\bf 129} (2009), 1698--1708.
\bibitem{Lehmer} D.H. Lehmer, {\it Factorization of certain cyclotomic functions}, 
  Ann. of Math. {\bf 34} (1933), 461--479.
\bibitem{Moss} M.J. Mossinghoff, {\it Algorithms for the determination of polynomials with small Mahler measure},
  Ph.D. Thesis, University of Texas at Austin, 1995.
\bibitem{MossWeb} M.J. Mossinghoff, website, {\it Lehmer's Problem}, {\tt http://www.cecm.sfu.ca/~mjm/Lehmer}.
\bibitem{MPV} M.J. Mossinghoff, C.G. Pinner and J.D. Vaaler, {\it Perturbing polynomials with all their roots on the unit circle},
  Math. Comp. {\bf 67} (1998), 1707--1726.
\bibitem{Northcott} D.G. Northcott, {\it An inequality on the theory of arithmetic on algebraic varieties},
	Proc. Cambridge Philos. Soc., {\bf 45} (1949), 502--509.
\bibitem{Samuels} C.L. Samuels, {\it The infimum in the metric Mahler measure}, Canad. Math. Bull., to appear.
\bibitem{Samuels2} C.L. Samuels, {\it The finiteness of computing the ultrametric Mahler measure}, Int. J. Number Theory, to appear.
\bibitem{Schinzel} A. Schinzel, {\it On the product of the conjugates outside the unit circle of an algebraic number},
 Acta Arith. {\bf 24} (1973), 385--399. Addendum, ibid.  {\bf 26} (1975), no. 3, 329--331.
\bibitem{Smyth} C.J. Smyth, {\it  On the product of the conjugates outside the unit circle of an algebraic integer},
 Bull. London Math. Soc.  {\bf 3}  (1971), 169--175.
\end{thebibliography}
\end{document}